\documentclass{amsart}

\usepackage{amsmath}
\usepackage{amssymb}
\usepackage{amsthm}
\usepackage{mathrsfs}

\DeclareMathOperator{\Gal}{Gal}
\DeclareMathOperator{\id}{id}

\DeclareMathOperator{\Ker}{Ker}
 
\DeclareMathOperator{\Fil}{Fil} 
\DeclareMathOperator{\Ind}{Ind} 

\newtheorem{thm}{Theorem}[section]
\newtheorem{prop}[thm]{Proposition}

\newtheorem{lem}[thm]{Lemma}

\newtheorem{rem}[thm]{Remark}

\begin{document}

\title
[Filtered modules]%
{Filtered modules corresponding to\\
potentially semi-stable representations}
\author{Naoki Imai}
\address{Research Institute for Mathematical Sciences, 
Kyoto University, Kyoto 606-8502 Japan}
\email{naoki@kurims.kyoto-u.ac.jp}


\maketitle

\begin{abstract}
We classify the filtered modules with coefficients 
corresponding to 
two-dimensional potentially semi-stable
$p$-adic representations 
of the absolute Galois groups of $p$-adic fields 
under the assumptions that 
$p$ is odd and the coefficients are large enough. 
\end{abstract}

\section*{Introduction}

Let $p$ be an odd prime number, 
and let $K$ be a $p$-adic field. 
The absolute Galois group of $K$ is 
denoted by $G_K$. 
By the fundamental theorem of Colmez and Fontaine \cite{CF}, 
there exists a correspondence between 
potentially semi-stable $p$-adic representations 
and admissible filtered $(\phi ,N)$-modules 
with Galois action. 
The aim of this paper is the 
classification of the admissible filtered 
$(\phi ,N)$-modules with Galois action 
corresponding to 
two-dimensional potentially semi-stable 
$p$-adic representations 
of $G_K$ with coefficients in a $p$-adic field $E$. 

If $K=\mathbb{Q}_p$ and 
$E= \mathbb{Q}_p$, 
the classification is given in 
\cite[Appendix A]{FM} under the assumption 
that $p \geq 5$. 
If $K=\mathbb{Q}_p$ and 
$E$ is general, 
these filtered $(\phi ,N)$-modules 
are studied in \cite{BM} and \cite{Sav}, 
and the classification is given by 
Ghate and M\'{e}zard in \cite{GM} 
under the assumptions that 
$p$ is odd and $E$ is large enough. 
In this paper, 
we generalize the results of \cite{GM} 
to the case where $K$ is a general $p$-adic field. 

In the case where $K$ is a general $p$-adic field, 
filtrations are determined by many weights 
and many elements of $\mathbb{P}^1 (E)$. 
In fact we need $[K:\mathbb{Q}_p ]$ elemens 
of $\mathbb{P}^1 (E)$ to parametrize 
two-dimensional potentially semi-stable
$p$-adic representations. 
These elements of $\mathbb{P}^1 (E)$ 
play a role similar to Fontaine-Mazur's 
$\mathfrak{L}$-invariants. 

After writing of this paper, 
the author has known that there is preceding research \cite{Do} 
on this subject by Dousmanis. 
The author does not claim priority, 
but there are some differences. 
In \cite{Do}, a classification is 
given by Frobenius action, 
and in this paper, 
we give a classification by Galois action. 
Let $F$ be a finite extension of $K$. 
A potentially semi-stable representation $\rho$ 
is said to be $F$-semi-stable, 
if the restriction of $\rho$ to 
the absolute Galois group of $F$ is semi-stable. 
In \cite{Do}, a classification of $F$-semi-stable 
representations is given 
for a general finite Galois extension $F$ of $K$. 
In this paper, 
we give a class of finite Galois extensions of $K$ 
such that any potentially semi-stable representation 
is $F$-semi-stable for a field $F$ in this class, 
and give a classification of $F$-semi-stable 
representations and 
a more explicit description of Galois action 
of $\Gal (F/K)$ for $F$ in this class, 
assuming $p \neq 2$. 
This difference is conspicuous in the supercuspidal case. 
Let $F_0$ be the maximal unramified extension 
of $\mathbb{Q}_p$ contained in $F$. 
In \cite[5.3]{Do}, 
it is proved that $\Gal (F/K)$-action on 
a filtered $(\phi.N)$-$(F_0 \otimes _{\mathbb{Q}_p} E)$-module 
comes from a $\Gal (F/K)$-action on 
the two-dimensional $E$-vector space 
in the supercuspidal case.  
In this paper, 
we study the $\Gal (F/K)$-action explicitly 
by using a structure of $\Gal (F/K)$, 
of coure, assumeing $F$ is in some class. 
Then, in this paper, we first fix a large enough 
coefficient field, and do not extend it 
in the classification. 

This paper is clearly influenced by the paper \cite{GM}, 
and we owe a lot of arguments to \cite{GM}. 
We mention it here, and do not repeat it 
each times in the sequel. 

\subsection*{Acknowledgment}
The author is supported by the Research Fellowships 
of the Japan Society for the Promotion 
of Science for Young Scientists. 
He would like to thank Gerasimos Dousmanis 
for permitting this paper. 
He is grateful to a referee for a careful reading of this paper 
and suggestions for improvements. 
\subsection*{Notation}

Throughout this paper, we use the following notation.
Let $p$ be an odd prime number, 
and $\mathbb{C}_p$ be the $p$-adic completion 
of the algebraic closure of $\mathbb{Q}_p$. 
Let $K$ be a $p$-adic field. 
We consider $K$ as a subfield of $\mathbb{C}_p$. 
The residue field of $K$ is 
denoted by $k$, 
whose cardinality is $q$. 
Let $K_0$ be the maximal unramified 
extension of $\mathbb{Q}_p$
contained in $K$. 
For any $p$-adic field $L$, 
the absolute Galois group of $L$ is 
denoted by $G_L$, 
the inertia subgroup of $G_L$ is 
denoted by $I_L$,  
the Weil group of $L$ is 
denoted by $W_L$, 
the ring of integers of $L$ 
is denoted by $\mathcal{O}_L$ and 
the unique maximal ideal of 
$\mathcal{O}_L$ is 
denoted by $\mathfrak{p}_L$. 
For a Galois extension $L$ of $K$, 
the inertia subgroup of $\Gal (L/K)$ is 
denoted by $I (L/K)$. 
Let $v_p$ be the valuations of 
$p$-adic fields normalized by $v_p (p)=1$. 

\section{Filtered $(\phi ,N)$-modules}

Let $E$ be a $p$-adic field. 
We consider a two-dimensional $p$-adic 
representation $V$ of $G_K$ over $E$, 
which is denoted by $\rho :G_K \to GL (V)$. 
As in \cite{Fon}, 
we can construct $K_0$-algebra $B_{\mathrm{st}}$ 
with a Frobenius endomorphism, a monodromy operator 
and Galois action. 
Further, we can define a decreasing filtration 
on $K \otimes_{K_0} B_{\mathrm{st}}$. 
Let $F$ be a finite Galois extension of $K$, 
and $F_0$ be the maximal unramified extension 
of $\mathbb{Q}_p$ contained in $F$. 
Then we have $B_{\mathrm{st}} ^{G_F} =F_0$. 
The $p$-adic representation 
$\rho$ is called 
$F$-semi-stable if and only if 
the dimension of 
$D_{\mathrm{st},F} (V)=(B_{\mathrm{st}} \otimes_{\mathbb{Q}_p} V)^{G_F}$ 
over $F_0$ is equal to the dimension of $V$ over $\mathbb{Q}_p$. 
If $\rho$ is $F$-semi-stable for 
some finite Galois extension $F$ of $K$, 
we say that $\rho$ is potentially semi-stable representation. 

Potentially semi-stable representations are 
Hodge-Tate. 
To fix a convention, we recall the definition 
of the Hodge-Tate weights. 
For $i \in \mathbb{Z}$, we put 
\[
 D_{\mathrm{HT}} ^i (V) =
 \bigl( \mathbb{C}_p (i) \otimes _{\mathbb{Q}_p } 
 V \bigr) ^{G_K } .
\]
Here and in the following, 
$(i)$ means $i$ times twists by 
the $p$-adic cyclotomic character of $G_{K}$. 
Then there is a $G_K$-equivariant isomorphism 
\[
 \bigoplus_{i \in \mathbb{Z}} 
 \mathbb{C}_p (-i) \otimes _K 
 D_{\mathrm{HT}} ^i (V) 
 \xrightarrow{\sim} 
 \mathbb{C}_p \otimes _{\mathbb{Q}_p} V
\]
of $(\mathbb{C}_p \otimes _{\mathbb{Q}_p} E)$-modules. 
The Hodge-Tate weights of the representation $V$ are 
the integers $i$ such that 
$D_{\mathrm{HT}} ^{-i} (V) \neq 0$, 
with multiplicities 
$\dim_E \bigl( D_{\mathrm{HT}} ^{-i} (V)\bigr)$. 

Next, we recall the definition of the filtered 
$\bigl( \phi ,N ,\Gal (F/K) ,E\bigr)$-modules. 
A filtered 
$\bigl( \phi ,N ,\Gal (F/K) ,E\bigr)$-module 
is a finite free 
$(F_0 \otimes _{\mathbb{Q}_p} E)$-module $D$ 
endowed with 
\begin{itemize}
 \item the Frobenius endomorphism: 
 an $F_0$-semi-linear, $E$-linear, bijective map $\phi :D \to D$, 
 \item the monodromy operator: 
 an $(F_0 \otimes _{\mathbb{Q}_p} E)$-linear, 
 nilpotent endomorphism $N: D \to D$ that satisfies 
 $N \phi =p\phi N$, 
 \item the Galois action: 
 an $F_0$-semi-linear, $E$-linear action of 
 $\Gal (F/K)$ that commutes with the action of 
 $\phi$ and $N$, 
 \item the filtration: 
 a decreasing filtration $(\Fil^i D_F )_{i \in \mathbb{Z}}$ 
 of $(F \otimes _{\mathbb{Q}_p} E)$-submodules of 
 $D_{F} =F \otimes_{F_0} D$ that are stable 
 under the action of $\Gal (F/K)$ and satisfy 
\[
 \Fil^i D_F =D_F \textrm{ for } i \ll 0 \textrm{ and } 
 \Fil^i D_F =0 \textrm{ for } i \gg 0. 
\]
\end{itemize}

Let $D$ be a filtered 
$\bigl( \phi ,N ,\Gal (F/K),E\bigr)$-module. 
Then, by forgetting the $E$-module structure, 
$D$ is also a filtered 
$\bigl( \phi ,N ,\Gal (F/K),\mathbb{Q}_p \bigr)$-module. 
We put $d=\dim_{F_0} D$. 
Then $\bigwedge _{F_0} ^d D$ is a filtered 
$\bigl( \phi ,N ,\Gal (F/K),\mathbb{Q}_p \bigr)$-module 
of dimension $1$ over $F_0$. 
We put 
\[
 t_{\mathrm{H}} (D) =\max \{i \in \mathbb{Z} \mid 
 \Fil^i (F \otimes_{F_0} {\bigwedge ^d}_{F_0} D) \neq 0 \},\
 t_{\mathrm{N}} (D) = v_p (\lambda)
\]
where $\lambda$ is an element of $F_0 ^{\times}$ 
that satisfies 
$\phi (x)=\lambda x$ 
for a non-zero element $x$ of 
$\bigwedge _{F_0} ^d D$. 
We say that $D$ is admissible 
if it satisfies the following two conditions: 
\begin{itemize}
 \item $t_{\mathrm{H}} (D) =t_{\mathrm{N}} (D)$. 
 \item For any $F_0$-submodule $D'$ of $D$ 
 that is stable by $\phi$ and $N$, 
 we have $t_{\mathrm{H}} (D') \leq t_{\mathrm{N}} (D')$, 
 where $D' _F \subset D_F$ is 
 equipped with the induced filtration. 
\end{itemize}
By \cite[Proposition 3.1.1.5]{BM}, 
we may replace the above second condition by the 
following condition: 
\begin{itemize}
 \item For any 
 $(F_0 \otimes _{\mathbb{Q}_p} E)$-submodule $D'$ of $D$ 
 that is stable by $\phi$ and $N$, 
 we have $t_{\mathrm{H}} (D') \leq t_{\mathrm{N}} (D')$, 
 where $D' _F \subset D_F$ is 
 equipped with the induced filtration. 
\end{itemize}

Let $k_0$ be a non-negative integer. 
By the results of \cite{CF}, 
there is an equivalence of categories between the 
category of two-dimensional $F$-semi-stable representations 
of $G_K$ over $E$ with Hodge-Tate weights in 
$\{0, \ldots ,k_0 \}$ and the category of 
admissible filtered 
$\bigl( \phi ,N ,\Gal (F/K),E\bigr)$-modules of rank $2$ over 
$F_0 \otimes_{\mathbb{Q}_p} E$ such that 
$\Fil^{-k_0 } (D_F )=D_F$ and $\Fil^1 (D_F )=0$. 
This equivalence of categories is given by 
the functor $D_{\mathrm{st},F}$ defined above. 
The aim of this paper is 
the classification of the objects of 
later categories under the assumption 
that $E$ is large enough. 

\section{Preliminaries}

Let $\rho :G_K \to GL (V)$ be 
a two-dimensional 
potentially-semi-stable representation 
over $E$. 
We assume that $\rho$ is $F$-semi-stable, 
and put $D =D_{\mathrm{st},F} (V)$. 
We recall the definition of 
Weil-Deligne representation associated to $\rho$. 
Now we have $W_K /W_F = \Gal (F/K)$. 
Let $m_0$ be the degree of 
the field extension $K_0$ over $\mathbb{Q}_p$. 
We define an $F_0$-linear action of $g \in W_K$ 
on $D$ by 
$(g \mod W_F )\circ \phi ^{-m_0 \alpha (g)}$, 
where the image of $g$ in 
$\Gal (\overline{k} /k)$ is the 
$\alpha (g)$-th power of the $q$-th power Frobenius map. 

We assume that $F_0 \subset E$. 
According to an isomorphism 
\[
 F_0 \otimes_{\mathbb{Q}_p} E 
 \xrightarrow{\sim} 
 \prod_{\sigma_i :F_0 \hookrightarrow E} E; 
 a \otimes b \mapsto \sigma_i (a)b,
\]
we have a decomposition 
\[
 D \xrightarrow{\sim} 
 \prod_{\sigma_i :F_0 \hookrightarrow E} 
 D_i . 
\] 
Here and in the sequel, 
$\sigma_i$ is an embedding determined by the $(-i)$-th power of 
the $p$-th power Frobenius map 
for $1 \leq i \leq [F_0 :\mathbb{Q}_p ]$. 
Then $D_i$, with an induced action of $W_K$ and 
an induced monodromy operator, 
defines a Weil-Deligne representation. 
The isomorphism class of 
this Weil-Deligne representation is 
independent of choice of $F$ and $\sigma_i$ 
(cf. \cite[Lemme 2.2.1.2]{BM}), 
and is, by definition, the Weil-Deligne representation 
$\mathrm{WD} (\rho)$ attached to $\rho$. 

We note that, in the above decomposition of $D$, 
the Frobenius endomorphism $\phi$ 
induce $E$-linear isomorphism 
$\phi :D_i \xrightarrow{\sim} D_{i+1}$. 
Naturally, we consider a suffix $i$ 
modulo $[F_0 :\mathbb{Q}_p ]$, and 
we often use such conventions 
in the sequel. 

A Galois type $\tau$ of degree $2$ is 
an equivalence class of representations 
$\tau :I_K \to GL_2 (\overline{\mathbb{Q}}_p )$ 
with open kernel that extend to representations 
of $W_K$. 
We say that 
an two-dimensional 
potentially semi-stable representation $\rho$ has 
Galois type $\tau$ if 
$\mathrm{WD} (\rho)|_{I_K} \simeq \tau$. 
The potentially semi-stable 
representation $\rho$ is 
$F$-semi-stable if and only if 
$\tau |_{I_{F}}$ is trivial. 

For a group $G$, an element $g \in G$, 
a normal subgroup $H$ of $G$ and 
a character 
$\chi : H\to \overline{\mathbb{Q}}_p ^{\times}$, 
we define a character 
$\chi^g : H\to \overline{\mathbb{Q}}_p ^{\times}$ 
by $\chi^g (h) =\chi (ghg^{-1} )$ for $h \in H$. 

\begin{lem}\label{form}
 Let $\tau$ be a Galois type of degree $2$. 
Then $\tau$ has one of the following forms: 
\begin{enumerate}
 \item 
 $\tau \simeq \chi_1 |_{I_K} \oplus \chi_2 |_{I_K}$, 
 where $\chi_1$, $\chi_2$ are characters of 
 $W_K$ finite on $I_K$, 
 \item 
 $\tau \simeq \Ind_{W_{K'}} ^{W_K} (\chi) |_{I_K}
 =\chi|_{I_K} \oplus \chi^{\sigma} |_{I_K}$, 
 where $K'$ is the unramified quadratic extension of $K$, 
 $\chi$ is a character of $W_{K'}$ 
 that is finite on $I_{K'}$ 
 and does not extend to $W_K$, and 
 $\sigma \in W_K$ is a lift of 
 the generator of $\Gal (K'/K)$, 
 \item 
 $\tau \simeq \Ind_{W_{K'}} ^{W_K} (\chi) |_{I_K}$, 
 where $K'$ is a ramified quadratic extension of $K$, 
 and $\chi$ is a character of $W_{K'}$ 
 such that $\chi$ is finite on $I_{K'}$ 
 and $\chi |_{I_{K'}}$ does not extend to $I_K$. 
\end{enumerate}
\end{lem}
\begin{proof}
This is a classical lemma, but we briefly recall a proof. 

We extend $\tau$ to a representation of $W_K$, 
which is denoted by $\tilde{\tau}$. 
If $\tilde{\tau}$ is reducible, we are in the case $(1)$, 
so we may assume that $\tilde{\tau}$ is irreducible. 

First, we treat the case where $\tau$ is reducible. 
In this case, 
$\tau \simeq \chi \oplus \chi'$ for 
some characters $\chi, \chi'$ of $I_K$. 
By irreducibility of $\tilde{\tau}$, 
we have $\chi' =\chi^{\sigma}$. 
Then $\tilde{\tau}|_{W_{K'}}$ is already reducible 
for the unramified quadratic extension $K'$ of $K$. 
So we are in the case $(2)$. 

Next, we treat the case where $\tau$ is irreducible. 
Let $I_K ^{\mathrm{w}}$ be the wild inertia subgroup of $I_K$. 
Then $\tau|_{I_K ^{\mathrm{w}}}$ is reducible, 
because a dimension of an irreducible representation of 
a $p$-group is a power of $p$ and $p \neq 2$. 
Then $\tilde{\tau}|_{W_{K'}}$ is already reducible 
for a ramified quadratic extension $K'$ of $K$. 
So we are in the case $(3)$. 
\end{proof}

To avoid the problem of the rationality, 
we assume that 
$E$ is a Galois extension over $\mathbb{Q}_p$, 
$F \subset E$ and the following: 

\begin{quote}
For all $p$-adic fields $K'$ such that 
$K \subset K' \subset F$ and $[K' :K] \leq 2$, 
and for all characters $\chi$ of $W_{K'}$ 
that are trivial on $I_{F}$, 
the restrictions $\chi|_{I_{K'}}$ 
factor through $E^{\times}$. 
\end{quote}
For example, 
if $E$ contains the $|I(F/K)|$-th roots of unity, 
then this condition is satisfied. 

In the sequel, let 
$\rho :G_K \to GL(V)$ be a two-dimensional 
potentially semi-stable representation 
over $E$ with 
Hodge-Tate weight in $\{0,\ldots ,k_0 \}$, 
and $\tau$ be its Galois type. 

\begin{lem}(cf. \cite[Lemma 2.3]{GM})
If $\rho$ is not potentially crystalline, 
then $\tau$ is a scalar. 
\end{lem}

Therefore, 
there are following three possibilities: 
\begin{itemize}
 \item Special or Steinberg case: 
 $N\neq 0$ and $\tau$ is a scalar. 
 \item Principal series case: 
 $N=0$ and $\tau$ is as in $(1)$ of 
 Lemma \ref{form}. 
 \item Supercuspidal case: 
 $N=0$ and $\tau$ is as in $(2)$ or $(3)$ 
 of Lemma \ref{form}.
\end{itemize}

Next, we study the structure of the filtrations. 
We assume $\rho$ is $F$-semi-stable, 
and take the corresponding filtered 
$\bigl( \phi ,N,\Gal (F/K),E\bigr)$-module $D$. 
We have a decomposition
\[
 F \otimes_{\mathbb{Q}_p} E \xrightarrow{\sim} 
 \prod_{j_F :F \hookrightarrow E} E = 
 \prod_{j:K \hookrightarrow E} 
 \Biggl( \prod_{j_F :F\hookrightarrow E,\,
 j_F |_K =j} E \Biggr) =
 \prod_{j:K \hookrightarrow E} E_j , 
\]
where $j_F$ and $j$ are $\mathbb{Q}_p$-embeddings 
and we put 
\[
 E_j =\prod_{j_F :F\hookrightarrow E,\, j_F |_K =j} E. 
\]
According to the above decomposition, 
we have decompositions
\[
 D_F \cong \prod_{j:K \hookrightarrow E} 
 D_{F,j} \textrm{ and }
 \Fil^i D_F \cong \prod_{j:K \hookrightarrow E} 
 \Fil_j ^i D_F .
\]
Because $\Fil^i D_F$ is $\Gal (F/K)$-stable, 
$\Fil_j ^i D_F$ is free over $E_j$. 
We take integers 
$0 \leq k_{j,1} \leq k_{j,2} \leq k_0$ 
such that 
\[
 D_{F,j} = \Fil_j ^{-k_{j,2}} D_F 
 \supsetneq 
 \Fil_j ^{1-k_{j,2}} D_F = \Fil_j ^{-k_{j,1}} D_F 
 \supsetneq 
 \Fil_j ^{1-k_{j,1}} D_F =0. 
\]
Then the Hodge-Tate weights of $\rho$ are 
$\bigcup_{j :K \hookrightarrow E} \{k_{j,1} ,k_{j,2} \}$. 

We are going to prepare some lemmas. 
\begin{lem}\label{isom}
There is a $\Gal (F/K)$-equivariant isomorphism 
\[
 F \otimes_{K} E 
 \xrightarrow{\sim} E_j
\]
of $E$-algebra. 
\end{lem}
\begin{proof}
Let $j_0$ be a natural inclusion $K\subset E$. 
Take an extension 
$j_E :E \xrightarrow{\sim} E$ 
of $j:K \hookrightarrow E$. 
Then a $\Gal (F/K)$-equivariant isomorphism 
\[
 \prod_{j_F :F\hookrightarrow E,\,
 j_F |_K =j_0 } E \xrightarrow{\sim} 
 \prod_{j_F :F\hookrightarrow E,\,
 j_F |_K =j} E 
\]
of $E$-algebra is given by sending 
$j_F$-components to 
$(j_E \circ j_F)$-components. 
\end{proof}

\begin{lem}\label{inv} 
If $k_{j,1} < k_{j,2}$, then 
$\Fil_j ^{-k_{j,1}} D_F \subset D_{F,j}$ is 
spanned by a Galois invariant element over $E_j$. 
\end{lem}
\begin{proof}
A generator of $\Fil_j ^{-k_{j,1}} D_F$ over $E_j$ 
generates an $E_j ^{\times}$-torsor with $\Gal (F/K)$-action. 
An $E_j ^{\times}$-torsor with $\Gal (F/K)$-action is tirivial, 
if $H^1 \bigl( \Gal (F/K), E_j ^{\times} \bigr) =0$. 
So it suffices to show that 
$H^1 \bigl( \Gal (F/K), E_j ^{\times} \bigr) =0$. 
By Lemma \ref{isom}, 
$E_j ^{\times}$ is isomorphic to 
$(F \otimes_{K} E)^{\times}$, 
and it is further isomorphic to 
$\Ind_{\{\id_F\}} ^{\Gal (F/K)} E^{\times}$. 
By Shapiro's lemma, 
$H^1 \bigl( \Gal (F/K),
 \Ind_{\{\id_F\}} ^{\Gal (F/K)} E^{\times} \bigr) 
 =H^1 (\{\id_F \} ,E^{\times} )=0$. 
\end{proof}

\begin{lem}\label{basis}
Let $K'$, $M$ be $p$-adic fields such that 
$K \subset K' \subset M \subset F$ and 
$M$ is a Galois extension of $K'$. 
Let $\chi :\Gal (M/K') \to E^{\times}$ be 
a character. 
We put $m=[K' :K]$. 
Then there exist 
$x_1 ,\ldots ,x_m \in M \otimes_K E$ 
that satisfy the followings: 
\begin{itemize}
 \item 
 For $x \in M \otimes_K E$, 
 we have $gx = \bigl( 1\otimes \chi(g)^{-1} \bigr) x$ 
 for all $g \in \Gal (M/K')$ if and only if 
 $x=\sum_{i=1} ^m (1\otimes a_i )x_i$ for $a_i \in E$. 
 \item 
 For $a_i \in E$, we have 
 $\sum_{i=1} ^m (1\otimes a_i )x_i \in (M \otimes_K E)^{\times}$ 
 if and only if $a_i \neq 0$ for all $i$. 
\end{itemize}
\end{lem}
\begin{proof}
We have a decomposition
\[
 M \otimes_K E \xrightarrow{\sim} 
 \prod_{j_M :M \hookrightarrow E} E = 
 \prod_{j' :K' \hookrightarrow E} 
 \Biggl( \prod_{j_M :M\hookrightarrow E,\,
 j_M |_{K'} =j'} E \Biggr) =
 \prod_{j' :K' \hookrightarrow E} E_{j'} , 
\]
where $j_M$ and $j'$ are $K$-embeddings and we put 
\[
 E_{j'} =\prod_{j_M :M\hookrightarrow E,\, j_M |_{K'} =j'} E. 
\]
Let $(x_{j'} )_{j'} \in \prod_{j' :K' \hookrightarrow E} E_{j'}$ 
be the image of $x$ 
under the above isomorphism. 
Then, 
$gx = \bigl( 1\otimes \chi(g)^{-1} \bigr) x$ 
for all $g \in \Gal (M/K')$ if and only if 
$gx_{j'} = \chi(g)^{-1} x_{j'}$ 
for all $g \in \Gal (M/K')$ and all 
$j':K' \hookrightarrow E$. 
Further, 
$x \in (M \otimes_K E)^{\times}$ if and only if 
$x_{j'} \in E_{j'} ^{\times}$ 
for all $j'$. 
As in the proof of Lemma \ref{isom}, 
we can show there is a $\Gal (M/K')$-equivariant 
isomorphism 
$M \otimes_{K'} E \xrightarrow{\sim} E_{j'}$ 
of $E$-algebra. 
So, to prove this Lemma, 
it suffices to treat the case where $m=1$. 

We assume that $m=1$. 
Take $\alpha \in M$ such that 
$g(\alpha)$ for $g \in \Gal (M/K)$ form 
a basis of $M$ over $K$. 
Then $x \in M \otimes_K E$ can be 
written uniquely as 
\[
 \sum_{g \in \Gal (M/K)}  g(\alpha) \otimes a_g 
\]
for $a_g \in E$. 
If $hx = \bigl( 1\otimes \chi(h)^{-1} \bigr) x$ 
for all $h \in \Gal (M/K)$, 
we have $a_{i,h^{-1} g} =\chi^{-1} (h)a_g$ 
for all $g,h \in \Gal (M/K)$. 
By putting $a_1 =a_{\id_M}$, we have 
\[
 x= (1 \otimes a_1 ) 
 \sum_{g\in \Gal (M/K)} 
 g(\alpha) \otimes \chi (g) .
\]
It suffices to put 
$x_1 =\sum_{g\in \Gal (M/K)} g(\alpha)\otimes \chi (g)$. 
\end{proof}

\section{Classification}

\subsection{Special or Steinberg case}

In this case, 
$\tau \simeq \chi |_{I_K} \oplus \chi |_{I_K}$ 
for some character $\chi$ of $W_K$ 
that is finite on $I_K$, 
and there exists a totally ramified cyclic 
extension $F$ of $K$ such that 
$\chi |_{I_F}$ is trivial. 
So we may assume that $\rho$ is 
$F$-semi-stable, and 
$\chi$ determine 
the action of $\Gal (F/K)$ on $D$, 
which is again denoted by $\chi$. 

Since $N\phi =p\phi N$, we have that 
$\Ker N$ is $\phi$-stable 
and free of rank $1$ over 
$F_0 \otimes_{\mathbb{Q}_p} E$. 
So we can take a basis $e_1 ,e_2$ of $D$ 
over $F_0 \otimes_{\mathbb{Q}_p} E$ such 
that $N(e_1 )=e_2$ and $N(e_2 )=0$. 
Again by $N\phi =p\phi N$, we must have 
$\phi (e_1 )=\frac{p}{\alpha} e_1 +\gamma e_2$ 
and $\phi (e_2 )=\frac{1}{\alpha} e_2$ with 
$\alpha \in (F_0 \otimes_{\mathbb{Q}_p} E)^{\times}$ 
and $\gamma \in F_0 \otimes_{\mathbb{Q}_p} E$. 
Modifying $e_1$ by a scalar multiple of $e_2$, 
we may assume $\gamma =0$. 
Let 
$(\alpha_i)_i \in 
 \prod_{\sigma_i :F_0 \hookrightarrow E} E$ 
be the image of $\alpha$ under the isomorphism 
\[
 F_0 \otimes_{\mathbb{Q}_p} E \xrightarrow{\sim} 
 \prod_{\sigma_i :F_0 \hookrightarrow E} E. 
\]
Then, by calculations, 
we have 
\begin{align*}
 t_{\mathrm{H}} (D) &=-[E:K]
 \sum_{j:K \hookrightarrow E}(k_{j,1} +k_{j,2} ), \\ 
 t_{\mathrm{N}} (D) &=[E:F_0 ]\biggl( 
 m_0 - 2\sum_i v_p (\alpha_i )\biggl). 
\end{align*}
So the condition 
$t_{\mathrm{H}} (D) = t_{\mathrm{N}} (D)$ is 
equivalent to that 
\[
 2[K:K_0 ]\sum_{i} v_p (\alpha_i) 
 = \sum_{j} (k_{j,1} +k_{j,2} +1). 
\]

For $j:K \hookrightarrow E$ satisfying 
$k_{j,1} < k_{j,2}$, by 
Lemma \ref{inv}, we take 
$a_j ,b_j \in E_j$ such that 
$\Fil_j ^{-k_{j,1}} D_F =E_j (a_j e_1 +b_j e_2)$, 
and $(a_j e_1 +b_j e_2)$ is 
$\Gal (F/K)$-invariant. 
We note that $a_j =0$ or $a_j \in E_j ^{\times}$ 
and that $b_j =0$ or $b_j \in E_j ^{\times}$. 

The only non-trivial 
$(\phi, N)$-stable 
$(F_0 \otimes _{\mathbb{Q}_p} E)$-submodule 
of $D$ is 
$D' _2 =(F_0 \otimes _{\mathbb{Q}_p} E)e_2$. 
By calculations, we have 
\begin{align*}
 t_{\mathrm{H}} (D' _2) &=-[E:K]
 \Biggl\{ \sum_{a_j =0} k_{j,1} + 
 \sum_{a_j \neq 0} k_{j,2} + 
 \sum_{k_{j,1} =k_{j,2} } k_{j,2} \Biggr\}, \\ 
 t_{\mathrm{N}} (D' _2) &=-[E:F_0 ] 
 \sum_i v_p (\alpha_i ). 
\end{align*}
So the condition 
$t_{\mathrm{H}} (D' _2 )\leq t_{\mathrm{N}} (D' _2 )$ 
is equivalent to that 
\[
 [K:K_0 ]\sum_{i} v_p (\alpha_i) 
 \leq 
 \sum_{a_j =0} k_{j,1} + 
 \sum_{a_j \neq 0} k_{j,2} + 
 \sum_{k_{j,1} =k_{j,2} } k_{j,2} . 
\]

Since $(a_j e_1 +b_j e_2 )$ is 
$\Gal (F/K)$-invariant, 
$g \in \Gal (F/K)$ acts on 
$a_j$ and $b_j$ by $\chi (g)^{-1}$. 
By Lemma \ref{isom} and Lemma \ref{basis}, 
there is $x_1 \in E_j$ such that 
$a_j =a_j ' x_1$ and $b_j =b_j ' x_1$ 
for $a_j ' ,b_j ' \in E$. 
Then, for $j$ such that 
$a_j \neq 0$, 
\[
 \Fil_j ^{-k_{j,1}} D_F = 
 E_j (a_j ' x_1 e_1 +b_j ' x_1 e_2 ) = 
 E_j (e_1 -\mathfrak{L}_j e_2 )
\]
for $\mathfrak{L}_j \in E$. 

\begin{prop}
We assume that $N \neq 0$. Then 
$\tau \simeq \chi |_{I_K} \oplus \chi |_{I_K}$ 
for some character $\chi$ of $W_K$ 
that is finite on $I_K$. 
If we take a totally ramified cyclic 
extension $F$ of $K$ such that 
$\chi$ is trivial on $I_F$, then 
$D=(F_0 \otimes_{\mathbb{Q}_p} E)e_1 \oplus 
 (F_0 \otimes_{\mathbb{Q}_p} E)e_2$ with 
\[ 
 N(e_1 )=e_2 ,\ N(e_2 )=0,\ 
 \phi (e_1 )=\frac{p}{\alpha} e_1 ,\ 
 \phi (e_2 )=\frac{1}{\alpha} e_2 
\]
for $\alpha \in (F_0 \otimes_{\mathbb{Q}_p} E)^{\times}$, 
\[
 ge_1 =\chi (g)e_1 ,\ ge_2 =\chi (g)e_2 
\]
for $g \in \Gal (F/K)$ and 
\[
 \Fil_j ^{-k_{j,1} } D_F =
 \begin{cases}
 E_j e_2 & \textrm{if } j \in I_1 , \\
 E_j (e_1 -\mathfrak{L}_j e_2 ) \textrm{ for }
 \mathfrak{L}_j \in E & \textrm{if } j \in I_2 
 \end{cases} 
\] 
for $j$ such that $k_{j,1} < k_{j,2}$, 
where 
\[
 2[K:K_0 ]\sum_{i} v_p (\alpha_i) 
 = \sum_{j} (k_{j,1} +k_{j,2} +1), 
\]
and $I_1 ,I_2 $ are any disjoint sets such that 
$I_1 \cup I_2 =\{j \mid k_{j,1} < k_{j,2}\}$ 
and 
\[
 [K:K_0 ]\sum_{i} v_p (\alpha_i) 
 \leq 
 \sum_{j \in I_1 } k_{j,1} +
 \sum_{j \in I_2 } k_{j,2} + 
 \sum_{k_{j,1} =k_{j,2} } k_{j,2} . 
\]
\end{prop}

\subsection{Principal series case}

In this case, 
$\tau \simeq \chi_1 |_{I_K} \oplus \chi_2 |_{I_K}$ 
and $N=0$. 
We can take a totally ramified abelian extension 
$F$ of $K$ such that 
$\chi_1 |_{I_F}$ and $\chi_2 |_{I_F}$ 
are trivial. 
Then $\chi_1$ and $\chi_2$ determine 
the action of $\Gal (F/K)$ on $D$, 
which is again denoted by the same symbols. 

\subsubsection{Irreducible case}
First, we assume that 
$\chi_1 |_{I_K} =\chi_2 |_{I_K}$ 
and $D$ has no non-trivial 
$\phi$-stable 
$(F_0 \otimes_{\mathbb{Q}_p} E)$-submodule. 
In this case, we say that $\phi$ is irreducible. 
If not, we say that $\phi$ is reducible. 
We put $\chi =\chi_1$. 

Take bases $e_{i,1} ,e_{i,2}$ of $D_i$ 
over $E$ for $1 \leq i \leq m_0$ 
so that 
\[
 \phi (e_{1,1} )=ae_{2,1} +ce_{2,2} ,\ 
 \phi (e_{1,2} )=be_{2,1} +de_{2,2} 
\]
for $a,b,c,d \in E$, and 
\[
 \phi (e_{i,1} )=e_{i+1,1} ,\ 
 \phi (e_{i,2} )=e_{i+1,2} 
\]
for $2 \leq i \leq m_0$. 
Let $e_1 ,e_2$ be a basis 
of $D$ over $F_0 \otimes_{\mathbb{Q}_p} E$ 
determined by 
$(e_{i,1} )_i$, $(e_{i,2} )_i$ 
under the isomorphism 
$D \xrightarrow{\sim} \prod_i D_i$. 
We will use the same notation in the classification of 
other cases.

Since $\phi$ is irreducible, 
$b \neq 0$ and $c \neq 0$. 
Modifying $e_{i,1}$ by 
a scalar multiple of $e_{i,2}$, 
we may assume $d=0$. 
If $X^2 -aX-bc$ is reducible in $E[X]$, 
by replacing the bases, we can see 
that $\phi$ is reducible. 
This is a contradiction. 
So $X^2 -aX-bc$ is irreducible in $E[X]$. 

Conversely, we suppose that 
$a,b,c \in E$ are given, $d=0$, and 
$X^2 -aX-bc$ is irreducible in $E[X]$. 
Then the above description determines 
an endomorphism $\phi$. 
We prove that this endomorphism $\phi$ 
is irreducible. 
If $\phi$ is reducible, 
there are $A_i \in GL_2 (E)$ such that 
\[
 A_2 ^{-1} 
 \begin{pmatrix}
  a & b \\ c & 0 
 \end{pmatrix} 
 A_1 ,\ A_3 ^{-1} A_2 ,\ A_4 ^{-1} A_3 ,
 \ldots ,\ A_1 ^{-1} A_{m_0} 
\] 
are all upper triangular matrices. 
Then, multiplying these matrices together, 
we have that 
$A_1 ^{-1} 
 \begin{pmatrix}
  a & b \\ c & 0 
 \end{pmatrix} A_1$ 
is an upper triangular matrix. 
This contradicts that 
$X^2 -aX-bc$ is irreducible in $E[X]$. 

As above, the endomorphism $\phi$ is given by 
$a,b,c \in E$ such that 
$X^2 -aX-bc$ is reducible in $E[X]$. 
Now, by calculation, we have 
\begin{align*}
 t_{\mathrm{H}} (D) &=-[E:K]
 \sum_{j:K \hookrightarrow E}(k_{j,1} +k_{j,2} ), \\ 
 t_{\mathrm{N}} (D) &=[E:F_0 ]
 \, v_p (bc). 
\end{align*}
So the condition 
$t_{\mathrm{H}} (D) = t_{\mathrm{N}} (D)$ is 
equivalent to that 
\[
 -[K:K_0 ]\, v_p (bc) 
 = \sum_{j} (k_{j,1} +k_{j,2} ). 
\]

Since $\phi$ is irreducible, 
$D$ has no non-trivial 
$(\phi,N)$-stable 
$(F_0 \otimes_{\mathbb{Q}_p} E)$-submodule. 
So there is no condition on the filtrations. 
For $j$ such that 
$k_{j,1} < k_{j,2}$, 
by Lemma \ref{isom}, Lemma \ref{inv} and 
Lemma \ref{basis}, we have 
\[
 \Fil_j ^{-k_{j,1}} D_F =E_j (a_j e_1 +b_j e_2 ) 
\]
for $(a_j ,b_j ) \in \mathbb{P}^1 (E)$. 

By studies of the other cases, 
$\phi$ is irreducible only if 
$N=0$ and $\tau \simeq \chi |_{I_K} \oplus \chi |_{I_K}$ 
for some character $\chi$ of $W_K$ 
that is finite on $I_K$. 

\begin{prop} 
We assume that $\phi$ is irreducible. 
Then 
$N=0$ and 
$\tau \simeq \chi |_{I_K} \oplus \chi |_{I_K}$ 
for some character $\chi$ of $W_K$ 
that is finite on $I_K$. 
If we take a totally ramified cyclic 
extension $F$ of $K$ such that 
$\chi$ is trivial on $I_F$, then 
$D=(F_0 \otimes_{\mathbb{Q}_p} E)e_1 \oplus 
 (F_0 \otimes_{\mathbb{Q}_p} E)e_2$ with 
\[
 \phi (e_{1,1} )=ae_{2,1} +ce_{2,2} ,\ 
 \phi (e_{1,2} )=be_{2,1} 
\]
for $a,b \in E^{\times}$ such that 
$X^2 -aX-bc$ is irreducible in $E[X]$, 
\[
 \phi (e_{i,1} )=e_{i+1,1} ,\ 
 \phi (e_{i,2} )=e_{i+1,2} 
\]
for $2 \leq i \leq m_0$, 
\[
 ge_1 =\chi (g)e_1 ,\ ge_2 =\chi (g)e_2 
\]
for $g \in \Gal (F/K)$ and, 
for $j$ such that $k_{j,1} < k_{j,2}$, 
\[
 \Fil_j ^{-k_{j,1} } D_F = 
 E_j (a_j e_1 +b_j e_2 ) 
\] 
for $(a_j ,b_j ) \in \mathbb{P}^1 (E)$, 
where 
\[
 -[K:K_0 ]\, v_p (bc) 
 = \sum_j (k_{j,1} +k_{j,2} ). 
\]
\end{prop}

\subsubsection{Non-split reducible case}

If $D$ has two or more non-trivial 
$\phi$-stable 
$(F_0 \otimes_{\mathbb{Q}_p} E)$-submodules, 
we say that $\phi$ is split. 
If not, we say that $\phi$ is non-split. 
We assume that $\chi_1 |_{I_K} =\chi_2 |_{I_K}$ 
and that $\phi$ is non-split and reducible. 
We put $\chi =\chi_1$. 

Since $\phi$ is reducible, 
we can take bases $e_{i,1} ,e_{i,2}$ of $D_i$ 
over $E$ 
and $a_i ,b_i ,d_i \in E$ 
for all $i$ so that 
\[
 \phi (e_{i,1} )=a_i e_{i+1,1} ,\ 
 \phi (e_{i,2} )=b_i e_{i+1,1} +d_i e_{i+1,2} 
\]
for all $i$. 
Replacing the bases, 
we may assume that $a_i =d_i =1$ 
and $b_i =0$ for $2 \leq i \leq n$. 
Since $\phi$ is non-split, 
$a_1 =d_1 \neq 0$ and $b_1 \neq 0$. 
We put $a=a_1$ and $b=b_1$. 

Conversely, we suppose that 
$a,b \in E^{\times}$ are given. 
Then the above description determines 
an endomorphism $\phi$. 
We prove that 
this endomorphism $\phi$ is non-split. 
If $\phi$ is split, 
there are $A_i \in GL_2 (E)$ such that 
\[
 A_2 ^{-1} 
 \begin{pmatrix}
  a & b \\ 0 & a 
 \end{pmatrix} 
 A_1 ,\ A_3 ^{-1} A_2 ,\ A_4 ^{-1} A_3 ,
 \ldots ,\ A_1 ^{-1} A_{m_0} 
\] 
are all diagonal matrices. 
Then, multiplying these matrices together, 
we have that 
$A_1 ^{-1} 
 \begin{pmatrix}
  a & b \\ 0 & a 
 \end{pmatrix} A_1$ 
is a diagonal matrix. 
This contradicts that 
$b \neq 0$. 

As above, the endomorphism $\phi$ is given by 
$a,b \in E^{\times}$. 
The condition 
$t_{\mathrm{H}} (D) = t_{\mathrm{N}} (D)$ is 
equivalent to that 
\[
 -2[K:K_0 ]\, v_p (a) 
 = \sum_{j} (k_{j,1} +k_{j,2} ). 
\]

Now we have bases $e_{i,1} ,e_{i,2}$ of $D_i$ 
over $E$ such that 
\[
 \phi (e_{1,1} )=ae_{2,1} ,\ 
 \phi (e_{1,2} )=be_{2,1} +ae_{2,2}
\]
for $a,b \in E^{\times}$, and 
\[
 \phi (e_{i,1} )=e_{i+1,1} ,\ 
 \phi (e_{i,2} )=e_{i+1,2} 
\]
for $2 \leq i \leq m_0$. 

For $j:K \hookrightarrow E$ satisfying 
$k_{j,1} < k_{j,2}$, by 
Lemma \ref{inv}, we take 
$a_j ,b_j \in E_j$ such that 
$\Fil_j ^{-k_{j,1}} D_F =E_j (a_j e_1 +b_j e_2)$, 
and $(a_j e_1 +b_j e_2)$ is $\Gal (F/K)$-invariant. 

The only non-trivial 
$(\phi, N)$-stable 
$(F_0 \otimes _{\mathbb{Q}_p} E)$-submodule 
of $D$ is 
$D' _1 =(F_0 \otimes _{\mathbb{Q}_p} E)e_1$. 
The condition 
$t_{\mathrm{H}} (D' _1 )\leq t_{\mathrm{N}} (D' _1 )$ 
is equivalent to that 
\[
 -[K:K_0 ]\, v_p (a) 
 \leq 
 \sum_{b_j =0} k_{j,1} +
 \sum_{b_j \neq 0} k_{j,2} + 
 \sum_{k_{j,1} =k_{j,2} } k_{j,2} . 
\]
As in the special or Steinberg case, 
for $j$ such that 
$b_j \neq 0$, 
\[
 \Fil_j ^{-k_{j,1}} D_F = 
 E_j (-\mathfrak{L}_j e_1 +e_2 ) 
\]
for $\mathfrak{L}_j \in E$. 

By studies of the other cases, 
$\phi$ is non-split reducible only if 
$N=0$ and 
$\tau \simeq \chi |_{I_K} \oplus \chi |_{I_K}$ 
for some character $\chi$ of $W_K$ 
that is finite on $I_K$. 

\begin{prop}
We assume that $\phi$ is 
non-split reducible. Then $N=0$ and 
$\tau \simeq \chi |_{I_K} \oplus \chi |_{I_K}$ 
for some character $\chi$ of $W_K$ 
that is finite on $I_K$. 
If we take a totally ramified cyclic 
extension $F$ of $K$ such that 
$\chi$ is trivial on $I_F$, then 
$D=(F_0 \otimes_{\mathbb{Q}_p} E)e_1 \oplus 
 (F_0 \otimes_{\mathbb{Q}_p} E)e_2$ with 
\[
 \phi (e_{1,1} )=ae_{2,1} ,\ 
 \phi (e_{1,2} )=be_{2,1} +ae_{2,2}
\]
for $a,b \in E^{\times}$, 
\[
 \phi (e_{i,1} )=e_{i+1,1} ,\ 
 \phi (e_{i,2} )=e_{i+1,2} 
\]
for $2 \leq i \leq m_0$, 
\[
 ge_1 =\chi (g)e_1 ,\ ge_2 =\chi (g)e_2 
\]
for $g \in \Gal (F/K)$ and 
\[
 \Fil_j ^{-k_{j,1} } D_F =
 \begin{cases}
 E_j e_1 & \textrm{if } j \in I_1 , \\
 E_j (-\mathfrak{L}_j e_1 +e_2 ) \textrm{ for }
 \mathfrak{L}_j \in E & \textrm{if } j \in I_2 
 \end{cases} 
\] 
for $j$ such that 
$k_{j,1} < k_{j,2}$, where 
\[
 -2[K:K_0 ]\, v_p (a) 
 = \sum_j (k_{j,1} +k_{j,2} ), 
\]
and $I_1 ,I_2 $ are any disjoint sets such that 
$I_1 \cup I_2 =\{j \mid k_{j,1} < k_{j,2}\}$ 
and 
\[
 -[K:K_0 ]\, v_p (a) 
 \leq 
 \sum_{j \in I_1 } k_{j,1} + 
 \sum_{j \in I_2 } k_{j,2} + 
 \sum_{k_{j,1} =k_{j,2} } k_{j,2} . 
\]
\end{prop}

\subsubsection{Split case}

The remaining cases are the following two cases: 
\begin{itemize}
 \item $\chi_1 |_{I_K} =\chi_2 |_{I_K}$ 
 and $\phi$ is split. 
 \item $\chi_1 |_{I_K} \neq \chi_2 |_{I_K}$. 
\end{itemize}

First, we assume that 
$\chi_1 |_{I_K} \neq \chi_2 |_{I_K}$. 
Let $e_1 ,e_2$ be a basis of $D$ over 
$F_0 \otimes_{\mathbb{Q}_p} E$ such that 
$\Gal (F/K)$ acts on $e_1$ by $\chi_1$ 
and $e_2$ by $\chi_2$. 
We put
\[
 \phi (e_1 )=\alpha e_1 +\gamma e_2 ,\ 
 \phi (e_2 )=\beta e_1 +\delta e_2 , 
\]
where 
$\alpha ,\beta ,\gamma ,\delta \in F_0 \otimes_{\mathbb{Q}_p} E$. 
Since $\phi$ commutes with the action of 
$\Gal (F/K)$ and 
$\chi_1 |_{I_K} \neq \chi_2 |_{I_K}$, 
we have $\beta =\gamma =0$. 
So, in the both cases, 
we may assume that $\phi$ is split. 

We take bases $e_{i,1} ,e_{i,2}$ of $D_i$ over $E$ 
so that 
\[
 \phi (e_{1,1} )=ae_{2,1} ,\ 
 \phi (e_{1,2} )=be_{2,2} 
\]
for some $a,b \in E^{\times}$ and 
\[
 \phi (e_{i,1} )=e_{i+1,1} ,\ 
 \phi (e_{i,2} )=e_{i+1,2} 
\]
for $2 \leq i \leq m_0$. 
Let $e_1 ,e_2$ be a basis 
of $D$ over $F_0 \otimes_{\mathbb{Q}_p} E$ 
determined by 
$(e_{i,1} )_i$, $(e_{i,2} )_i$ 
under the isomorphism 
$D \xrightarrow{\sim} \prod_i D_i$. 

Then the condition 
$t_{\mathrm{H}} (D) = t_{\mathrm{N}} (D)$ is 
equivalent to that 
\[
 [K:K_0 ]\, v_p (ab) 
 = \sum_{j} (k_{j,1} +k_{j,2} ). \tag{$S$} 
\]

For $j:K \hookrightarrow E$ satisfying 
$k_{j,1} < k_{j,2}$, by 
Lemma \ref{inv}, we take 
$a_j ,b_j \in E_j$ such that 
$\Fil_j ^{-k_{j,1}} D_F =E_j (a_j e_1 +b_j e_2)$, 
and $(a_j e_1 +b_j e_2)$ is $\Gal (F/K)$-invariant. 

Since $(a_j e_1 +b_j e_2 )$ is 
$\Gal (F/K)$-invariant, 
$g \in \Gal (F/K)$ acts on 
$a_j$ and $b_j$ by 
$\chi_1 (g)^{-1}$ and 
$\chi_2 (g)^{-1}$ respectively. 
By Lemma \ref{isom} and Lemma \ref{basis}, 
there are $x_1 ,x_2 \in E_j$ such that 
$a_j =a_j ' x_1$ and $b_j =b_j ' x_2$ 
for $a_j ' ,b_j ' \in E$. 
Then, for $j$ such that 
$a_j \neq 0$ and $b_j \neq 0$, 
we have 
\[
 \Fil_j ^{-k_{j,1}} D_F = 
 E_j (a_j ' x_1 e_1 +b_j ' x_2 e_2 ) = 
 E_j (e_1 -\mathfrak{L}_j x_0 e_2 ) 
\]
for $\mathfrak{L}_j \in E^{\times}$, 
where we put $x_0 =x_1 ^{-1} x_2$. 

If $a \neq b$, the non-trivial 
$(\phi, N)$-stable 
$(F_0 \otimes _{\mathbb{Q}_p} E)$-submodules 
of $D$ are 
$D' _1 =(F_0 \otimes _{\mathbb{Q}_p} E)e_1$ and 
$D' _2 =(F_0 \otimes _{\mathbb{Q}_p} E)e_2$. 
The condition 
$t_{\mathrm{H}} (D' _1 )\leq t_{\mathrm{N}} (D' _1 )$ 
is equivalent to that 
\[
 [K:K_0 ]\, v_p (a) 
 \leq 
 \sum_{b_j =0} k_{j,1} + 
 \sum_{b_j \neq 0} k_{j,2} + 
 \sum_{k_{j,1} =k_{j,2} } k_{j,2} . 
\]
The condition 
$t_{\mathrm{H}} (D' _2 )\leq t_{\mathrm{N}} (D' _2 )$ 
is equivalent to that 
\[
 [K:K_0 ]\, v_p (b) 
 \leq 
 \sum_{a_j =0} k_{j,1} + 
 \sum_{a_j \neq 0} k_{j,2} + 
 \sum_{k_{j,1} =k_{j,2} } k_{j,2} . 
\]

If $a=b$, the non-trivial 
$(\phi, N)$-stable 
$(F_0 \otimes _{\mathbb{Q}_p} E)$-submodules 
of $D$ are 
$D' _1$, $D' _2$ and 
$D' _{\mathfrak{L}} =(F_0 \otimes _{\mathbb{Q}_p} E) 
 (e_1 -\mathfrak{L} e_2 )$ 
for $\mathfrak{L} \in E^{\times}$. 
For $\mathfrak{L} \in E^{\times}$, 
the condition 
$t_{\mathrm{H}} (D' _{\mathfrak{L}} ) 
 \leq t_{\mathrm{N}} (D' _{\mathfrak{L}} )$ 
is equivalent to that 
\begin{align*}
 [K:K_0 ]\, v_p (a) 
 \leq 
 &\sum_{a_j b_j =0} k_{j,2} + 
 \sum_{k_{j,1} =k_{j,2} } k_{j,2} \tag{$S_{\mathfrak{L}}$} \\ 
 &+\sum_{a_j b_j \neq 0} 
 \bigl\{ t_j (\mathfrak{L} ,\mathfrak{L}_j )k_{j,1} 
 +\bigl( 1-t_j (\mathfrak{L} ,\mathfrak{L}_j ) \bigr)
 k_{j,2} \bigr\}, 
\end{align*}
where 
\[
 t_j (\mathfrak{L} ,\mathfrak{L}_j )= 
 \frac{\bigl| \{ j_F :F \hookrightarrow E \mid 
 j_F \textrm{-component of } 
 \mathfrak{L}_j x_0 \in E_j 
 \textrm{ is } \mathfrak{L} \} \bigr| } 
 {[F:K]}. 
\]
If $t_j (\mathfrak{L} ,\mathfrak{L}_j ) \leq 1/2$, 
the condition $(S_{\mathfrak{L}})$ is 
automatically satisfied by the condition $(S)$. 

We assume that $t_j (\mathfrak{L} ,\mathfrak{L}_j ) > 1/2$. 
Then we have 
\[
 \frac{\Bigl| \Ker \Bigl(\chi_1 \chi_2 ^{-1} : \Gal (F/K) \to 
 \overline{\mathbb{Q}}_p ^{\times} \Bigr) \Bigr|}{[F:K]} 
 > \frac{1}{2} , 
\]
because $\Gal (F/K)$ act on $x_0$ by 
$\chi_1 \chi_2 ^{-1}$. 
This implies that 
$\chi_1 |_{I_K} =\chi_2 |_{I_K}$ 
and 
\[
 x_0 =(x_E )_{j_F} \in 
 \prod_{j_F :F\hookrightarrow E,\, j_F |_K =j} E 
\] 
for some $x_E \in E^{\times}$. 
Then 
$\mathfrak{L}_j x_E =\mathfrak{L}$ 
and 
$t_j (\mathfrak{L} ,\mathfrak{L}_j )=1$. 

\begin{prop}
We assume that $N=0$ and $\phi$ is 
split reducible and 
$\tau \simeq \chi_1 |_{I_K} \oplus \chi_2 |_{I_K}$ 
for some character $\chi_1, \chi_2$ of $W_K$ 
that are finite on $I_K$. 
If we take a totally ramified cyclic 
extension $F$ of $K$ such that 
$\chi_1, \chi_2$ is trivial on $I_F$, then 
$D=(F_0 \otimes_{\mathbb{Q}_p} E)e_1 \oplus 
 (F_0 \otimes_{\mathbb{Q}_p} E)e_2$ with 
\[
 \phi (e_{1,1} )=ae_{2,1} ,\ 
 \phi (e_{1,2} )=be_{2,2} 
\]
for $a,b \in E^{\times}$ and 
\[
 \phi (e_{i,1} )=e_{i+1,1} ,\ 
 \phi (e_{i,2} )=e_{i+1,2} 
\]
for $2 \leq i \leq m_0$ and 
\[
 \Fil_j ^{-k_{j,1} } D_F =
 \begin{cases}
 E_j e_1 & \textrm{if } j \in I_1 , \\
 E_j e_2 & \textrm{if } j \in I_2 , \\
 E_j (e_1 -\mathfrak{L}_j x_{0} e_2 ) \textrm{ for }
 \mathfrak{L}_j \in E^{\times} & \textrm{if } j \in I_3 
 \end{cases} 
\] 
for $j$ such that 
$k_{j,1} < k_{j,2}$, where 
\[
 [K:K_0 ]\, v_p (ab) 
 = \sum_j (k_{j,1} +k_{j,2} ), 
\]
and $I_1 ,I_2 ,I_3 $ are any disjoint sets such that 
$I_1 \cup I_2 \cup I_3 =
 \{j \mid k_{j,1} < k_{j,2}\}$ 
and 
\begin{align*} 
 [K:K_0 ]\, v_p (a) 
 &\leq 
 \sum_{j \in I_1 } k_{j,1} +
 \sum_{j \in I_2 \cup I_3 } k_{j,2} + 
 \sum_{k_{j,1} =k_{j,2} } k_{j,2} , \\
 [K:K_0 ]\, v_p (b) 
 &\leq 
 \sum_{j \in I_2 } k_{j,1} +
 \sum_{j \in I_1 \cup I_3 } k_{j,2} + 
 \sum_{k_{j,1} =k_{j,2} } k_{j,2} , 
\end{align*}
and, if $a=b$ and $\chi_1 |_{I_K} =\chi_2 |_{I_K}$, further 
\begin{align*}
 [K:K_0 ]\, v_p (a) 
 \leq 
 \sum_{j \in I_3 ,\, \mathfrak{L}_j x_E =\mathfrak{L}} k_{j,1} + 
 \sum_{j \in I_3 ,\, \mathfrak{L}_j x_E \neq \mathfrak{L} } k_{j,2} 
 + \sum_{j \in I_1 \cup I_2} k_{j,2} 
 + \sum_{k_{j,1} =k_{j,2}} k_{j,2} 
\end{align*}
for all $\mathfrak{L} \in E^{\times}$. 
\end{prop}

\subsection{Supercuspidal case}

In this case, 
$N=0$ and 
$\tau \simeq \Ind_{W_{K'}} ^{W_K} (\chi) |_{I_K}$ 
for a quadratic extension $K'$ of $K$ and 
a character $\chi$ of $W_{K'}$ that is 
finite on $I_{K'}$. 
Let $k'$ be the 
residue field of $K'$. 
We take a totally ramified abelian extension 
$L$ of $K'$ such that 
$\chi |_{I_L}$ is trivial. 

For a uniformizer $\pi'$ of $K'$ 
and a positive integer $n$, 
let $K' _{\pi' ,n}$ be 
the Lubin-Tate extension of $K'$ 
generated by the ${\pi' }^n$-torsion points. 
For any $p$-adic field $M$ and 
a positive integer $n$, 
we put $U_M ^{(n)} =1+\mathfrak{p}_M ^n$. 
Then we have 
\[
 \Gal (K' _{\pi' ,n} /K') \cong 
 (\mathcal{O}_{K'} /\mathfrak{p}_{K'} ^n )^{\times} 
 \cong {k' }^{\times} \times 
 \bigl( U_{K'} ^{(1)} /U_{K'} ^{(n)} \bigr). 
\] 
For any $p$-adic field $M$ and 
a positive integer $m$, 
let $M _m$ be the unramified 
extension of $M$ of degree $m$. 

\subsubsection{Unramified case}

We first treat the case in 
$(2)$ of Lemma \ref{form}, where 
$K'$ is unramified over $K$ 
and $\chi$ does not extend 
to $W_K$. 
We take a uniformizer $\pi$ of $K$. 
This is also a uniformizer of $K'$. 
We take positive integers $m_1$ and $n_1$ 
so that $L$ is contained in 
$K' _{m_1} K' _{\pi , {n_1}}$, 
and put $F=K' _{m_1} K' _{\pi , {n_1}}$. 
Then $\rho$ is crystalline over $F$, 
and $F$ is a Galois extension of $K$. 
     
We put $f(X) =\pi X +X^{q^2}$. 
For a positive integer $n$, 
let $f^{(n)} (X)$ be the $n$-th 
iterate of $f(X)$. 
We take a root 
$\theta$ of $f^{(n_1)} (X)$ 
in $K' _{\pi ,n_1 }$ 
that is not a root of 
$f^{(n_1-1)} (X)$. 
Then $K' _{\pi ,n_1} =K' (\theta)$. 
We can see that 
$K (\theta)$ is 
a totally ramified extension of $K$ 
and that 
$F$ is an unramified extension of 
$K(\theta)$ of degree $2m_1$. 
Now the restriction 
$\Gal \bigl( F/K(\theta)\bigr) \to \Gal (K' _{m_1 } /K)$ 
is an isomorphism, 
and $\Gal (F/K)$ is a semi-direct product 
of $\Gal \bigl( F/K(\theta)\bigr)$ by 
$\Gal (F/K' _{m_1 })$. 
We take a generator $\sigma$ of 
$\Gal (F/K(\theta))$. 
Then the restriction $\sigma|_{K'}$ 
is the non-trivial element of 
$\Gal (K' /K)$. 

We consider a decomposition 
\[
 U_{K'} ^{(1)} /U_{K'} ^{(n_1 )} = 
 U_{n_1 ,+}
 \times 
 U_{n_1 ,-} 
\]
of abelian groups 
such that 
$\sigma (\gamma_1 )=\gamma_1$ for 
$\gamma_1 \in U_{n_1 ,+}$ 
and 
$\sigma (\gamma_2 )=\gamma_2 ^{-1}$ for 
$\gamma_2 \in U_{n_1 ,-}$. 
There is an exact sequence
\[
 1 \to U_K ^{(1)} /U_K ^{(n_1 )} \to 
 U_{K'} ^{(1)} /U_{K'} ^{(n_1 )} \to 
 U_{K'} ^{(1)} /U_{K'} ^{(n_1 )} 
\] 
where  
the first map is induced from 
a natural inclusion 
and the second map is induced from 
a map 
\[
 U_{K'} ^{(1)} \to U_{K'} ^{(1)} ;\ 
 g \mapsto \sigma(g)g^{-1}. 
\]
Then, by the above exact sequence, we see that 
\[
 U_{n_1 ,+} \cong U_K ^{(1)} /U_K ^{(n_1 )} ,\ 
 U_{n_1 ,-} \cong U_{K'} ^{(1)} /
 \bigl(U_K ^{(1)} U_{K'} ^{(n_1 )} \bigr) 
\]
and $|U_{n_1 ,+} |=|U_{n_1 ,-} | =q^{n_1 -1}$. 

Now, the restriction 
$\Gal (F/K' _{m_1 }) \to \Gal (K' _{\pi ,n_1 } /K')$ 
is an isomorphism. 
Then we can prove that, under an identification 
\[
 \Gal (F/K' _{m_1 }) 
 \cong 
 \Gal (K' _{\pi ,n_1 } /K') 
 \cong {k' }^{\times} \times 
 U_{n_1 ,+}
 \times 
 U_{n_1 ,-} , 
\]
we have 
\[ 
 \sigma^{-1} \delta \sigma =\delta^q ,\ 
 \sigma^{-1} \gamma_1 \sigma =\gamma_1 
 \textrm{ and }
 \sigma^{-1} \gamma_2 \sigma =\gamma_2 ^{-1} \tag{$\ast$} 
\] 
for $\delta \in {k'}^{\times}$, 
$\gamma_1 \in U_{n_1 ,+}$ 
and 
$\gamma_2 \in U_{n_1 ,-}$. 

Considering $\chi |_{I_K}$ as 
a character of 
\[
 I (F/K) 
 \cong {k' }^{\times} \times 
 U_{n_1 ,+}
 \times 
 U_{n_1 ,-} , 
\]
we write 
$\chi =\omega^s \cdot \chi_1 \cdot \chi_2$, 
where $\omega$ is the Teichm\"{u}ller character, 
$s$ is an integer, 
and $\chi_1$ and $\chi_2$ are characters of 
$U_{n_1 ,+}$ 
and 
$U_{n_1 ,-}$ 
respectively. 
The condition that $\chi$ does not extend to 
$W_K$ is equivalent to that 
$\chi \neq \chi^{\sigma}$ on $W_{K'}$, 
and it is further equivalent to that 
$\chi \neq \chi^{\sigma}$ on $I_{K'}$. 
This last condition is equivalent to that 
$s \not\equiv 0 \mod q+1$ or 
$\chi_2 ^2 \neq 1$. 

Now we have $[F_0 :\mathbb{Q}_p ]=2m_0 m_1$. 
We take bases $e_{i,1}$, $e_{i,2}$ of $D_i$ 
over $E$ for $1 \leq i \leq 2m_0 m_1$ so that 
\begin{align*}
 \delta e_{i,1} &=\omega^s (\delta) e_{i,1} ,& 
 \gamma_1 e_{i,1} &=\chi_1 (\gamma_1 )e_{i,1} ,&
 \gamma_2 e_{i,1} &=\chi_2 (\gamma_2 ) e_{i,1} , \\ 
 \delta e_{i,2} &=\omega^{qs} (\delta) e_{i,2} ,& 
 \gamma_1 e_{i,2} &=\chi_1 (\gamma_1 )e_{i,2} ,&
 \gamma_2 e_{i,2} &=\chi_2 (\gamma_2 )^{-1} e_{i,2} 
\end{align*}
for $\delta \in {k'}^{\times}$, 
$\gamma_1 \in U_{n_1 ,+}$ 
and 
$\gamma_2 \in U_{n_1 ,-}$. 

\begin{rem}
A normalization of bases here is different from 
that in \cite[3.3.2]{GM}. 
We prefer that the action of $\delta$ on 
$e_{i,1}$, $e_{i,2}$ is the same form for all $i$. 
In stead of this, the action of $\sigma$ 
does not preserve lines generated by 
$e_1$ and $e_2$ as we see in the below. 
\end{rem}

Since $\sigma$ takes $D_i$ to $D_{i+m_0}$, 
we have that 
\[
 \sigma e_{i,1} =a_{i+m_0} e_{i+m_0 ,2}, \ 
 \sigma e_{i,2} =b_{i+m_0} e_{i+m_0 ,1}
\]
for some $a_{i+m_0} ,b_{i+m_0} \in E^{\times}$ by $(\ast)$. 
Because $\sigma ^{2m_1} =1$, we see that 
\[
 \prod_{l=1} ^{m_1} (a_{i+2lm_0 -m_0} 
 b_{i+2lm_0} ) =1 
\]
for all $i$. 
Replacing $e_{i,1}$ and $e_{i,2}$ by 
their scalar multiples, 
we may assume that 
\[
 \sigma e_{i,1} =e_{i+m_0 ,2}, \ 
 \sigma e_{i,2} =e_{i+m_0 ,1} . 
\]

Since $\phi$ takes $D_i$ to $D_{i+1}$ 
and commutes with the action of $I (F/K)$, 
we have that
\[
 \phi (e_{i,1} )=\frac{1}{\alpha_{i+1}} e_{i+1,1} , \ 
 \phi (e_{i,2} )=\frac{1}{\beta_{i+1}} e_{i+1,2} 
\] 
for some $\alpha_{i+1} ,\beta_{i+1} \in E^{\times}$ 
for all $i$. 
Since $\phi$ commutes with 
the action of $\sigma$, 
we have $\alpha_i =\beta_{i+m_0}$ 
and $\beta_i =\alpha_{i+m_0}$ for all $i$. 
Replacing $e_{i,1}$ and $e_{i,2}$ by their scalar 
multiples, we may further assume that 
$\alpha_i =\beta_i =1$ for 
$2 \leq i \leq m_0$. 

Let $e_1$, $e_2$ be a basis of 
$D$ over $F_0 \otimes_{\mathbb{Q}_p} E$ 
determined by $(e_{i,1} )_i$, 
$(e_{i,2} )_i$ under the isomorphism 
$D \xrightarrow{\sim} \prod_i D_i$. 
Then $\sigma e_1 =e_2$ and 
$\sigma e_2 =e_1$. 

The condition 
$t_{\mathrm{H}} (D) = t_{\mathrm{N}} (D)$ is 
equivalent to that 
\[
 [K:K_0 ]\, v_p (\alpha_1 \beta_1 ) 
 = \sum_{j} (k_{j,1} +k_{j,2} ). \tag{$U$} 
\]

For $j:K \hookrightarrow E$ satisfying 
$k_{j,1} < k_{j,2}$, by 
Lemma \ref{inv}, we take 
$a_j ,b_j \in E_j$ such that 
$\Fil_j ^{-k_{j,1}} D_F =E_j (a_j e_1 +b_j e_2)$, 
and $(a_j e_1 +b_j e_2)$ is 
$\Gal (F/K)$-invariant. 
By $\sigma (a_j e_1 +b_j e_2 )=(a_j e_1 +b_j e_2 )$, 
we get $\sigma (a_j )=b_j$ and 
$\sigma (b_j )=a_j$. 
So $a_j \in E_j ^{\times}$ if and only if 
$b_j \in E_j ^{\times}$. 

Since $(a_j e_1 +\sigma (a_j) e_2 )$ is 
$\Gal (F/K)$-invariant, 
$\sigma^2 (a_j )=a_j$ and 
$g \in I(F/K)$ acts on 
$a_j$ by $\chi (g)^{-1}$. 
We prove that there are 
$x_{j,1} ,x_{j,2} \in E_j$ 
such that 
\begin{itemize}
 \item 
 $a_j$ satisfies the above condition if and only if 
 $a_j =a_{j,1} x_{j,1} +a_{j,2} x_{j,2}$ 
 for some $a_{j,1} ,a_{j,2} \in E$, 
 \item 
 for $a_{j,1} ,a_{j,2} \in E$, we have 
 $a_{j,1} x_{j,1} +a_{j,2} x_{j,2} \in E_j ^{\times}$ 
 if and only if $a_{j,1} \neq 0$ and $a_{j,2} \neq 0$. 
\end{itemize} 
By Lemma \ref{isom}, we may replace 
$E_j$ by $F \otimes_K E$. 
Then $\sigma^2 (a_j )=a_j$ if and only if 
$a_j \in K' _{\pi,n_1} \otimes_K E$. 
By Lemma \ref{basis}, we get the claim. 
We put 
$x_j (a_{j,1} ,a_{j,2}) =a_{j,1} x_{j,1} +a_{j,2} x_{j,2}$ 
and 
$x_j ^{\sigma} (a_{j,1} ,a_{j,2}) =
 \sigma \bigl( x_j (a_{j,1} ,a_{j,2} )\bigr)$. 
Then we have 
\[
 \Fil_j ^{-k_{j,1}} D_F = 
 E_j \bigl(x_j (a_{j,1} ,a_{j,2}) e_1 + 
 x_j ^{\sigma}(a_{j,1} ,a_{j,2}) e_2 \bigr) 
\]
for $(a_{j,1} ,a_{j,2}) \in \mathbb{P}^1 (E)$. 

The non-trivial 
$(\phi, N)$-stable 
$(F_0 \otimes _{\mathbb{Q}_p} E)$-submodules 
of $D$ are 
$D' _1 =(F_0 \otimes _{\mathbb{Q}_p} E)e_1$, 
$D' _2 =(F_0 \otimes _{\mathbb{Q}_p} E)e_2$ and 
$D' _{\mathfrak{L}} =(F_0 \otimes _{\mathbb{Q}_p} E) 
 (e_1 -\mathfrak{L} e_2 )$ 
for $\mathfrak{L} \in (F_0 \otimes _{\mathbb{Q_p}} E)^{\times}$ 
satisfying the following:
\begin{quote}
If $\mathfrak{L}$ corresponds to $(\mathfrak{L}_i)_i$ 
under the isomorphism 
\[
F_0 \otimes_{\mathbb{Q}_p} E \xrightarrow{\sim} 
\prod_{\sigma_i :F_0 \hookrightarrow E} E,
\]
then $\mathfrak{L}_{i+1} = \frac{\alpha_{i+1} }{\beta_{i+1} } \mathfrak{L}_i$ 
for all i. 
\end{quote}
The condition 
$t_{\mathrm{H}} (D' _1 )\leq t_{\mathrm{N}} (D' _1 )$ 
is equivalent to that 
\[
 [K:K_0 ]\, v_p (\alpha_1) 
 \leq 
 \sum_{a_{j,1} a_{j,2} =0} \frac{k_{j,1} +k_{j,2} }{2} 
 +\sum_{a_{j,1} a_{j,2} \neq 0} k_{j,2} 
 +\sum_{k_{j,1} =k_{j,2} } k_{j,2} ,
\]
the condition 
$t_{\mathrm{H}} (D' _2 )\leq t_{\mathrm{N}} (D' _2 )$ 
is equivalent to that 
\[
 [K:K_0 ]\, v_p (\beta_1) 
 \leq 
 \sum_{a_{j,1} a_{j,2} =0} \frac{k_{j,1} +k_{j,2} }{2} 
 +\sum_{a_{j,1} a_{j,2} \neq 0} k_{j,2} 
 +\sum_{k_{j,1} =k_{j,2} } k_{j,2} ,
\]
and the condition 
$t_{\mathrm{H}} (D' _{\mathfrak{L}} ) 
 \leq t_{\mathrm{N}} (D' _{\mathfrak{L}} )$ 
is equivalent to that 
\begin{align*}
 [K:K_0 ]\, &\frac{v_p (\alpha_1 \beta_1)}{2} 
 \leq 
 \sum_{a_{j,1} a_{j,2} =0} k_{j,2} + 
 \sum_{k_{j,1} =k_{j,2} } k_{j,2} \tag{$U_{\mathfrak{L}}$} \\ 
 &+\sum_{a_{j,1} a_{j,2} \neq 0} 
 \Bigl\{ t_j \bigl(\mathfrak{L} ,(a_{j,1} ,a_{j,2})\bigr) k_{j,1} 
 +\Bigl( 1-t_j \bigl(\mathfrak{L} ,(a_{j,1} ,a_{j,2}) \bigr) \Bigr)
 k_{j,2} \Bigr\}, 
\end{align*}
where 
\[
 t_j \bigl(\mathfrak{L} ,(a_{j,1} ,a_{j,2}) \bigr)= 
 \frac{\Bigl| \bigl\{ j_F :F \hookrightarrow E \bigm| 
 j_F \textrm{-component of } 
 \frac{x_j ^{\sigma} (a_{j,1} ,a_{j,2} )}
 {x_j (a_{j,1} ,a_{j,2})} \in E_j 
 \textrm{ is } -\! \mathfrak{L}_{j_F} \bigr\} \Bigr| } 
 {[F:K]}. 
\]
Here and in the sequel, $\mathfrak{L}_{j_F}$ is the 
$j_F$-component of 
$\mathfrak{L} \in F_0 \otimes _{\mathbb{Q}_p} E 
\subset F \otimes _{\mathbb{Q}_p} E$. 
If 
$t_j \bigl(\mathfrak{L} ,(a_{j,1} ,a_{j,2}) \bigr) \leq 1/2$, 
the condition $(U_{\mathfrak{L}})$ is automatically 
satisfied by the condition $(U)$. 

To prove that 
$t_j \bigl(\mathfrak{L} ,(a_{j,1} ,a_{j,2}) \bigr) \leq 1/2$, 
we assume that 
$t_j \bigl(\mathfrak{L} ,(a_{j,1} ,a_{j,2}) \bigr) > 1/2$. 
We consider a decomposition 
\[
 E_j =\prod_{j_F :F\hookrightarrow E,\, j_F |_K =j} E 
 =\prod_{j_{F_0} :F_0 \hookrightarrow E,\, j_{F_0} |_K =j} 
 \Biggl( 
 \prod_{j_F :F\hookrightarrow E,\, j_F |_{F_0} =j_{F_0}} 
 E \Biggr) . 
\]
Then there is 
$j_{F_0} :F_0 \hookrightarrow E$ 
such that $j_{F_0} |_K =j$ and 
\[
 \frac{\biggl| \Bigl\{ j_F :F \hookrightarrow E \Bigm| 
 j_F |_{F_0} =j_{F_0} \textrm{ and } 
 j_F \textrm{-component of } 
 \frac{x_j ^{\sigma} (a_{j,1} ,a_{j,2} )}
 {x_j (a_{j,1} ,a_{j,2})} \in E_j 
 \textrm{ is } -\! \mathfrak{L}_{j_F} \Bigr\} \biggr| } 
 {[F:F_0 ]} 
\]
is greater than $1/2$. 
Here $\mathfrak{L}_{j_F}$ is independent of 
$j_F$ such that $j_F |_{F_0} =j_{F_0}$, 
because 
$\mathfrak{L} \in F_0 \otimes _{\mathbb{Q}_p} E$. 
Then we have 
\[
 \frac{\Bigl| \Ker \Bigl(\chi (\chi^{\sigma} )^{-1} : I(F/K) \to 
 \overline{\mathbb{Q}}_p ^{\times} \Bigr) \Bigr|}{[F:F_0 ]} 
 > \frac{1}{2} , 
\]
because $I (F/K' )$ act on 
$x_j ^{\sigma} (a_{j,1} ,a_{j,2} ) \big/ \bigl( x_j (a_{j,1} ,a_{j,2}) \bigr)$ 
by 
$\chi (\chi^{\sigma} )^{-1}$. 
This implies that 
$\chi |_{I_{K'}} =\chi^{\sigma} |_{I_{K'}}$, 
and contradicts the condition that 
$\chi$ does not extend to $W_K$. 
Thus we have proved that 
$t_j \bigl(\mathfrak{L} ,(a_{j,1} ,a_{j,2}) \bigr) \leq 1/2$. 

\begin{prop} 
We assume 
$\tau \simeq \Ind_{W_{K'}} ^{W_K} (\chi) |_{I_K}$ 
for the unramified quadratic extension $K'$ of $K$ and 
a character $\chi$ of $W_{K'}$ that is 
finite on $I_{K'}$ 
and does not extend to $W_K$. 
We take a uniformizer $\pi$ of $K$ 
and a totally ramified abelian extension $L$ of $K'$ 
such that $\chi$ is trivial on $I_L$, 
and take positive integers $m_1$ and $n_1$ 
so that $L$ is contained in 
$K' _{m_1} K' _{\pi , {n_1}}$. 
We put $F=K' _{m_1} K' _{\pi , {n_1}}$. 
Then $N=0$ and 
$D=(F_0 \otimes_{\mathbb{Q}_p} E)e_1 \oplus 
 (F_0 \otimes_{\mathbb{Q}_p} E)e_2$ with 
\begin{align*} 
 \phi (e_{i,1} )&=\frac{1}{\alpha_1} e_{i+1,1} ,& 
 \phi (e_{i,2} )&=\frac{1}{\beta_1} e_{i+1,2} ,& 
 &\textrm{if } i \equiv 0 & &\hspace*{-3em} \pmod{2m_0} , \\ 
 \phi (e_{i,1} )&=\frac{1}{\beta_1} e_{i+1,1} ,& 
 \phi (e_{i,2} )&=\frac{1}{\alpha_1} e_{i+1,2} ,& 
 &\textrm{if } i \equiv m_0 & &\hspace*{-3em}\pmod{2m_0} , \\ 
 \phi (e_{i,1} )&=e_{i+1,1} ,& 
 \phi (e_{i,2} )&=e_{i+1,2} ,& 
 &\textrm{if } i \not\equiv 0 & &\hspace*{-3em} \pmod{m_0}
\end{align*}
for $\alpha_1 ,\beta_1 \in E^{\times}$, 
\[
 \sigma e_1 =e_2 ,\ \sigma e_2 =e_1 ,\ 
 ge_1 =\bigl( 1 \otimes \chi (g) \bigr)e_1 ,\ 
 ge_2 =\bigl( 1 \otimes \chi^\sigma (g) \bigr)e_2 
\]
for $g \in I(F/K)$ and, 
for $j$ such that 
$k_{j,1} < k_{j,2}$,
\[
 \Fil_j ^{-k_{j,1}} D_F = 
 E_j \bigl(x_j (a_{j,1} ,a_{j,2}) e_1 + 
 x_j ^{\sigma} (a_{j,1} ,a_{j,2}) e_2 \bigr) 
\] 
for $(a_{j,1} ,a_{j,2}) \in \mathbb{P}^1 (E)$ 
where 
\[
 [K:K_0 ]\, v_p (\alpha_1 \beta_1 ) 
 = \sum_{j} (k_{j,1} +k_{j,2} ) 
\]
and 
\[
 \sum_j k_{j,1} + 
 \sum_{a_{j,1} a_{j,2} =0} \frac{k_{j,2} -k_{j,1} }{2}
 \leq 
 [K:K_0 ]\, v_p (\alpha_1) 
 \leq 
 \sum_j k_{j,2} - 
 \sum_{a_{j,1} a_{j,2} =0} \frac{k_{j,2} -k_{j,1} }{2}.  
\]
The definition of $\sigma$ is in the above discussion. 
\end{prop}

\subsubsection{Ramified case}

Next, we treat the case in 
$(3)$ of Lemma \ref{form}, 
where 
$K'$ is ramified over $K$ 
and $\chi |_{I_{K'}}$ does not extend 
to $I_K$. 

Let $\iota_0$ be the non-trivial 
element of $\Gal (K' /K)$. 
We take a uniformizer $\pi'$ of $K'$ 
such that $\iota_0 (\pi' )=-\pi'$. 
Then we have 
$(K' _{\pi' ,n} )^{\iota} =K' _{-\pi' ,n}$ 
for a positive integer $n$ 
and any lift $\iota \in G_K$ 
of $\iota_0$. 
So $K' _{\pi' ,n} K' _{-\pi' ,n}$ 
is a Galois extension of $K$. 
By the class field theory, 
the abelian extensions 
$K' _{\pi' ,n}$ and $K' _{-\pi' ,n}$ 
of $K'$ 
correspond to 
$\langle \pi' \rangle \times (1+\mathfrak{p}_{K'} ^n )$ 
and 
$\langle -\pi' \rangle \times (1+\mathfrak{p}_{K'} ^n )$ 
respectively. 
Then the abelian extension 
$K' _{\pi' ,n} K' _{-\pi' ,n}$ 
of $K'$ corresponds to 
$\langle {\pi'}^2 \rangle \times (1+\mathfrak{p}_{K'} ^n )$. 
So we see that 
$K' _{\pi' ,n} K' _{-\pi' ,n} =K' _2 K' _{\pi' ,n}$. 

We take positive integers $m_1$ and $n_1$ 
so that $L$ is contained in 
$K' _{2m_1} K' _{\pi' , {2n_1 +1}}$, 
and put $F=K' _{2m_1} K' _{\pi' , {2n_1 +1}}$. 
Then $F$ is a Galois extension of $K$, 
and $\rho$ is crystalline over $F$ 
because $\tau |_{I_F}$ is trivial. 

We consider an exact sequence 
\[
 1 \to \Gal (F/K' ) \to 
 \Gal (F/K) \to 
 \Gal (K' /K) \to 1. \tag{$\diamondsuit$} 
\]
Since the restriction 
$\Gal (F/K' _{2m_1} )\to \Gal (K' _{\pi' ,2n_1 +1} /K' )$ 
is an isomorphism, 
\begin{align*}
 \Gal (F/K' )&=\Gal (F/K' _{\pi' ,2n_1 +1} ) 
 \times \Gal (F/K' _{2m_1} ) \\ 
 &\cong 
 \Gal (F/K' _{\pi' ,2n_1 +1} ) 
 \times {k' }^{\times} \times 
 \bigl( U_{K'} ^{(1)} /U_{K'} ^{(2n_1 +1)} \bigr). 
\end{align*}
Let $\sigma$ be a generator of 
$\Gal (F/K_{\pi' ,2n_1 +1} )$, 
and $\delta_0$ be a generator 
of ${k'} ^{\times}$. 

We prove that the exact sequence 
$(\diamondsuit)$ does not split. 
We assume there is a lift 
$\iota \in \Gal (F/K)$ of $\iota_0$ 
such that $\iota^2 =1$. 
By multiplying $\iota$ 
by an element of 
$\Gal (F/K' _{\pi' ,2n_1 +1} ) \subset \Gal (F/K' )$, 
we may assume that $\iota \in I(F/K)$. 
Let $P (F/K)$ be 
the wild ramification subgroup of $I(F/K)$, 
and $I^{\mathrm{t}} (F/K)$ be 
the tame quotient group of $I(F/K)$. 
Let $\Bar{\iota}$ be the image of 
$\iota$ in $I^{\mathrm{t}} (F/K)$. 
If $\Bar{\iota} \neq 1$, 
we multiply $\iota$ by the element 
$\delta_0 ^{(q-1)/2}$ of 
${k'}^{\times} \subset \Gal (F/K' _{2m_1 })$. 
Then we have $\iota \in P(F/K)$, 
but this contradicts that $p \neq 2$. 
Thus we have proved the claim. 

For any lift $\iota \in \Gal (F/K)$, 
we have $\iota^2 \in \Gal (F/K' )$. 
Since the exact sequence $(\diamondsuit)$ 
does not split and $p \neq 2$, 
multiplying $\iota$ 
by an element of $\Gal (F/K' )$, 
we may assume that $\iota^2 =\delta_0$ 
and $\iota \in I(F/K)$. 
We fix this lift $\iota$ in the sequel. 

We consider a decomposition 
\[
 U_{K'} ^{(1)} /U_{K'} ^{(2n_1 +1)} = 
 U_{2n_1 +1 ,+}
 \times 
 U_{2n_1 +1,-} 
\]
of abelian groups 
such that 
$\iota_0 (\gamma_1 )=\gamma_1$ for 
$\gamma_1 \in U_{2n_1 +1,+}$ 
and 
$\iota_0 (\gamma_2 )=\gamma_2 ^{-1}$ for 
$\gamma_2 \in U_{2n_1 +1,-}$. 
There is an exact sequence
\[
 1 \to U_K ^{(1)} /U_K ^{(n_1 +1)} \to 
 U_{K'} ^{(1)} /U_{K'} ^{(2n_1 +1)} \to 
 U_{K'} ^{(1)} /U_{K'} ^{(2n_1 +1)} , 
\] 
where  
the first map is induced from 
a natural inclusion 
and the second map is induced from 
a map 
\[
 U_{K'} ^{(1)} \to U_{K'} ^{(1)} ;\ 
 g \mapsto \iota_0 (g)g^{-1}. 
\]
Then, by the above exact sequence, we see that 
\[
 U_{2n_1 +1 ,+} \cong U_K ^{(1)} /U_K ^{(n_1 +1)} ,\ 
 U_{2n_1 +1 ,-} \cong U_{K'} ^{(1)} /
 \bigl(U_K ^{(1)} U_{K'} ^{(2n_1 +1)} \bigr) 
\]
and $|U_{2n_1 +1,+} |=|U_{2n_1 +1,-} | =q^{n_1}$. 

We can prove that, under an identification 
\[
 \Gal (F/K' _{2m_1 }) 
 \cong 
 \Gal (K' _{\pi' ,2n_1 +1} /K') 
 \cong {k' }^{\times} \times 
 U_{2n_1 +1,+}
 \times 
 U_{2n_1 +1,-} , 
\]
we have 
\[ 
 \iota^{-1} \delta \iota =\delta ,\ 
 \iota^{-1} \gamma_1 \iota =\gamma_1 
 \textrm{ and }
 \iota^{-1} \gamma_2 \iota =\gamma_2 ^{-1} 
\] 
for $\delta \in {k'}^{\times}$, 
$\gamma_1 \in U_{2n_1 +1,+}$ 
and 
$\gamma_2 \in U_{2n_1 +1,-}$. 

Since $K' _{\pi' ,2n_1 +1}$ is not 
a normal extension of $K$, 
we have $\iota^{-1} \sigma \iota \neq \sigma$. 
We put 
$K'' =K' _{\pi' ,2n_1 +1} K' _{-\pi' ,2n_1 +1}$. 
Then $\sigma^2$ is a generator of 
$\Gal (F/K'' )$, and 
$\iota$ determines an automorphism 
of $K''$. So we have 
$\iota^{-1} \sigma^2 \iota =\sigma^2$. 
Since $\sigma^{-1} \iota^{-1} \sigma \iota$ 
is an element of $\Gal (F/K')$ of order $2$ 
and fixes $K_{2m_1}$, 
it is $\delta_0 ^{(q-1)/2}$. 
Hence we have 
\[
 \iota^{-1} \sigma \iota =\sigma \delta_0 ^{(q-1)/2} . 
 \tag{$\star$} 
\]

Considering $\chi |_{I_{K'}}$ as 
a character of 
\[
 I (F/K') 
 \cong {k' }^{\times} \times 
 U_{2n_1 +1,+}
 \times 
 U_{2n_1 +1,-} , 
\]
we write 
$\chi =\omega^s \cdot \chi_1 \cdot \chi_2$, 
where $\omega$ is the Teichm\"{u}ller character, 
$s$ is an integer, 
and $\chi_1$ and $\chi_2$ are characters of 
$U_{2n_1 +1,+}$ 
and 
$U_{2n_1 +1,-}$ 
respectively. 
The condition $\chi$ does not extend to 
$I_K$ is equivalent to that 
$\chi \neq \chi^{\iota}$ on $I_{K'}$, 
and it is further equivalent to 
that $\chi_2 ^2 \neq 1$. 

Now we have $[F_0 :\mathbb{Q}_p ]=2m_0 m_1$. 
We take bases $e_{i,1}$, $e_{i,2}$ of $D_i$ 
over $E$ for $1 \leq i \leq 2m_0 m_1$ so that 
\begin{align*}
 \iota e_{i,1} &=e_{i,2} ,& 
 \delta e_{i,1} &=\omega^s (\delta) e_{i,1} ,& 
 \gamma_1 e_{i,1} &=\chi_1 (\gamma_1 )e_{i,1} ,&
 \gamma_2 e_{i,1} &=\chi_2 (\gamma_2 ) e_{i,1} , \\ 
 \iota e_{i,2} &=\omega^s (\delta_0 )e_{i,1} ,& 
 \delta e_{i,2} &=\omega^s (\delta) e_{i,2} ,& 
 \gamma_1 e_{i,2} &=\chi_1 (\gamma_1 )e_{i,2} ,&
 \gamma_2 e_{i,2} &=\chi_2 (\gamma_2 )^{-1} e_{i,2} 
\end{align*}
for $\delta \in {k'}^{\times}$, 
$\gamma_1 \in U_{n_1 ,+}$ 
and 
$\gamma_2 \in U_{n_1 ,-}$. 

Since $\sigma$ takes $D_i$ to $D_{i+m_0}$, 
as in the unramified case, 
we may assume that 
$\sigma e_{i,1} =e_{i+m_0 ,1}$. 
Then we have that 
$\sigma e_{i,2} =(-1)^s e_{i+m_0 ,2}$ 
by $(\star)$. 

Since $\phi$ takes $D_i$ to $D_{i+1}$ 
and commutes with the action of $I (F/K)$, 
we have that
\[
 \phi (e_{i,1} )=\frac{1}{\alpha_{i+1}} e_{i+1,1} , \ 
 \phi (e_{i,2} )=\frac{1}{\alpha_{i+1}} e_{i+1,2} 
\] 
for some $\alpha_{i+1} \in E^{\times}$ 
for all $i$. 
Further, since $\phi$ commutes with 
the action of $\sigma$, 
we have $\alpha_i =\alpha_{i+m_0}$ for all $i$. 
Replacing $e_{i,1}$ and $e_{i,2}$ by their scalar 
multiples, we may further assume that 
$\alpha_i =1$ for 
$2 \leq i \leq m_0$. 

Let $e_1$, $e_2$ be a basis of 
$D$ over $F_0 \otimes_{\mathbb{Q}_p} E$ 
determined by $(e_{i,1} )_i$, 
$(e_{i,2} )_i$ under the isomorphism 
$D \xrightarrow{\sim} \prod_i D_i$. 
Then $\sigma e_1 =e_1$ and 
$\sigma e_2 =(-1)^s e_2$. 

The condition 
$t_{\mathrm{H}} (D) = t_{\mathrm{N}} (D)$ is 
equivalent to that 
\[
 2[K:K_0 ]\, v_p (\alpha_1 ) 
 = \sum_{j} (k_{j,1} +k_{j,2} ). \tag{$R$} 
\]

For $j:K \hookrightarrow E$ satisfying 
$k_{j,1} < k_{j,2}$, by 
Lemma \ref{inv}, we take 
$a_j ,b_j \in E_j$ such that 
$\Fil_j ^{-k_{j,1}} D_F =E_j (a_j e_1 +b_j e_2)$, 
and $(a_j e_1 +b_j e_2)$ is $\Gal (F/K)$-invariant. 
By $\iota (a_j e_1 +b_j e_2 )=(a_j e_1 +b_j e_2 )$, 
we get $\iota (a_j )=b_j$ and 
$\iota (b_j ) \omega^s (\delta_0 )=a_j$. 
So $a_j \in E_j ^{\times}$ if and only if 
$b_j \in E_j ^{\times}$. 

Since $\bigl(a_j e_1 +\iota (a_j) e_2 \bigr)$ is 
$\Gal (F/K)$-invariant, 
$\sigma (a_j )=a_j$ and 
$g \in I(F/K' )$ acts on 
$a_j$ by $\chi (g)^{-1}$. 
We prove that there are 
$x_{j,1} ,x_{j,2} \in E_j$ 
such that 
\begin{itemize}
 \item 
 $a_j$ satisfies the above condition if and only if 
 $a_j =a_{j,1} x_{j,1} +a_{j,2} x_{j,2}$ 
 for some $a_{j,1} ,a_{j,2} \in E$, 
 \item 
 for $a_{j,1} ,a_{j,2} \in E$, we have 
 $a_{j,1} x_{j,1} +a_{j,2} x_{j,2} \in E_j ^{\times}$ 
 if and only if $a_{j,1} \neq 0$ and $a_{j,2} \neq 0$. 
\end{itemize} 
By Lemma \ref{isom}, we may replace 
$E_j$ by $F \otimes_K E$. 
Then $\sigma (a_j )=a_j$ if and only if 
$a_j \in K' _{\pi' ,2n_1 +1} \otimes_K E$. 
By Lemma \ref{basis}, we get the claim. 
We put 
$x_j (a_{j,1} ,a_{j,2}) =a_{j,1} x_{j,1} +a_{j,2} x_{j,2}$ 
and 
$x_j ^{\iota} (a_{j,1} ,a_{j,2}) 
 =\iota \bigl(x_j (a_{j,1} ,a_{j,2} )\bigr)$. 
Then we have 
\[
 \Fil_j ^{-k_{j,1}} D_F = 
 E_j \bigl(x_j (a_{j,1} ,a_{j,2}) e_1 + 
 x_j ^{\iota} (a_{j,1} ,a_{j,2}) e_2 \bigr) 
\]
for $(a_{j,1} ,a_{j,2}) \in \mathbb{P}^1 (E)$. 

The non-trivial 
$(\phi, N)$-stable 
$(F_0 \otimes _{\mathbb{Q}_p} E)$-submodules 
of $D$ are 
$D' _1 =(F_0 \otimes _{\mathbb{Q}_p} E)e_1$, 
$D' _2 =(F_0 \otimes _{\mathbb{Q}_p} E)e_2$ and 
$D' _{\mathfrak{L}} =(F_0 \otimes _{\mathbb{Q}_p} E) 
 (e_1 -\mathfrak{L} e_2 )$ 
for $\mathfrak{L} \in E^{\times}$. 
The condition 
$t_{\mathrm{H}} (D' _1 )\leq t_{\mathrm{N}} (D' _1 )$ 
is equivalent to that 
\[
 [K:K_0 ]\, v_p (\alpha_1 ) 
 \leq 
 \sum_{a_{j,1} a_{j,2} =0} \frac{k_{j,1} +k_{j,2} }{2} 
 +\sum_{a_{j,1} a_{j,2} \neq 0} k_{j,2} 
 +\sum_{k_{j,1} =k_{j,2} } k_{j,2} , 
\]
and this condition is automatically satisfied 
by the condition ($R$). 
The condition 
$t_{\mathrm{H}} (D' _2 )\leq t_{\mathrm{N}} (D' _2 )$ 
is also equivalent to the same condition. 
For $\mathfrak{L} \in E^{\times}$, 
the condition 
$t_{\mathrm{H}} (D' _{\mathfrak{L}} ) 
 \leq t_{\mathrm{N}} (D' _{\mathfrak{L}} )$ 
is equivalent to that 
\begin{align*}
 [K:K_0 ]&\, v_p (\alpha_1 ) 
 \leq 
 \sum_{a_{j,1} a_{j,2} =0} k_{j,2} + 
 \sum_{k_{j,1} =k_{j,2} } k_{j,2} \tag{$R_{\mathfrak{L}}$} \\ 
 &+\sum_{a_{j,1} a_{j,2} \neq 0} 
 \Bigl\{ t_j \bigl(\mathfrak{L} ,(a_{j,1} ,a_{j,2})\bigr) k_{j,1} 
 +\Bigl( 1-t_j \bigl(\mathfrak{L} ,(a_{j,1} ,a_{j,2}) \bigr) \Bigr)
 k_{j,2} \Bigr\}, 
\end{align*} 
where 
\[
 t_j \bigl(\mathfrak{L} ,(a_{j,1} ,a_{j,2}) \bigr)= 
 \frac{\Bigl| \bigl\{ j_F :F \hookrightarrow E \bigm| 
 j_F \textrm{-component of } 
 \frac{x_j ^{\iota} (a_{j,1} ,a_{j,2} )} 
 {x_j (a_{j,1} ,a_{j,2})} \in E_j 
 \textrm{ is } -\! \mathfrak{L} \bigr\} \Bigr| } 
 {[F:K]}. 
\]
As in the unramified case, 
we can prove that 
$t_j \bigl(\mathfrak{L} ,(a_{j,1} ,a_{j,2}) \bigr) \leq 1/2$, 
using the condition that $\chi \neq \chi^{\iota}$ 
on $I_{K'}$. 
So the condition $(R_{\mathfrak{L}} )$ 
is automatically satisfied by 
the condition $(R)$. 

\begin{prop} 
We assume 
$\tau \simeq \Ind_{W_{K'}} ^{W_K} (\chi) |_{I_K}$ 
for a ramified quadratic extension $K'$ of $K$ and 
a character $\chi$ of $W_{K'}$ such that
$\chi|_{I_{K'}}$ is finite and 
does not extend to $I_K$. 
We take a uniformizer $\pi'$ of $K'$ 
and a totally ramified abelian extension $L$ of $K'$ 
such that $\chi$ is trivial on $I_L$, 
and take positive integers $m_1$ and $n_1$ 
so that $L$ is contained in 
$K' _{2m_1} K' _{\pi' , {2n_1 +1}}$. 
We put $F=K' _{2m_1} K' _{\pi' , {2n_1 +1}}$. 
Then $N=0$ and 
$D=(F_0 \otimes_{\mathbb{Q}_p} E)e_1 \oplus 
 (F_0 \otimes_{\mathbb{Q}_p} E)e_2$ with  
\begin{align*} 
 \phi (e_{i,1} )&=\frac{1}{\alpha_1} e_{i+1,1} ,& 
 \phi (e_{i,2} )&=\frac{1}{\alpha_1} e_{i+1,2} ,& 
 &\textrm{if } i \equiv 0 & &\hspace*{-3em} \pmod{m_0} , \\ 
 \phi (e_{i,1} )&=e_{i+1,1} ,& 
 \phi (e_{i,2} )&=e_{i+1,2} ,& 
 &\textrm{if } i \not\equiv 0 & &\hspace*{-3em} \pmod{m_0}
\end{align*}
for $\alpha_1 \in E^{\times}$, 
\begin{align*} 
 \sigma e_1 &=e_1 ,& \iota e_1 &=e_2 ,& 
 ge_1 &=\bigl( 1 \otimes \chi (g) \bigr)e_1 ,\\ 
 \sigma e_2 &=(-1)^s e_2, & 
 \iota e_2 &=\bigl( 1\otimes \omega^s (\delta_0 )\bigr)e_1 ,& 
 ge_2 &=\bigl( 1 \otimes \chi^\sigma (g) \bigr)e_2 
\end{align*} 
for $s \in \mathbb{Z}$ and $g \in I(F/K' )$ and, 
for $j$ such that 
$k_{j,1} < k_{j,2}$,
\[
 \Fil_j ^{-k_{j,1}} D_F = 
 E_j \bigl(x_j (a_{j,1} ,a_{j,2}) e_1 + 
 x_j ^{\iota} (a_{j,1} ,a_{j,2}) e_2 \bigr) 
\] 
for $(a_{j,1} ,a_{j,2}) \in \mathbb{P}^1 (E)$ 
where 
\[
 2[K:K_0 ]\, v_p (\alpha_1 ) 
 = \sum_{j} (k_{j,1} +k_{j,2} ). 
\] 
Here $\omega:k' \to \mathcal{O}_{K'} ^{\times}$ 
is the Teichm\"{u}ller character, 
and the definitions of $\sigma, \iota, \delta_0$ are in the above discussion. 
\end{prop}

\end{document}